\documentclass{article}
\usepackage{amsmath,amssymb,amsthm,stmaryrd,makeidx}
\usepackage[all]{xy}

\DeclareMathOperator{\enm}{End}
\DeclareMathOperator{\ext}{Ext}

\DeclareMathOperator{\id}{id}

\DeclareMathOperator{\hmm}{Hom}
\DeclareMathOperator{\mor}{Mor}

\DeclareMathOperator{\spec}{Spec}
\DeclareMathOperator{\aspec}{aSpec}

\DeclareMathOperator{\simp}{Simp}

\newcommand{\cat}[1]{\mathbf{#1}}

\newcommand{\grmcat}[1]

\input xypic

\newtheorem{proposition}{Proposition}
\newtheorem{theorem}{Theorem}
\newtheorem{lemma}{Lemma}
\newtheorem{corollary}{Corollary}
\newtheorem{definition}{Definition}

\newtheorem{remark}{Remark}

\newtheorem{example}{Example}

\makeindex

\begin{document}
\author{Arvid Siqveland}
\title{Associative Schemes}

\maketitle

\begin{abstract} 
Let $k$ be a field and let $A$ be an associative unital $k$-algebra. For a set of $r>0$ right $A$-modules $M=\{M_1,\dots,M_r\}$ we define the ring of locally defined functions $O^A(M).$ We then define the set $\aspec A$ of aprime modules which contains the simple modules, and give it a topology. For each open $U\subseteq\aspec A$ we let the ring $O^A(U)$ of functions defined on $U$ be the projective limit over inclusions of finite ordered subsets of $U$. This defines a presheaf of rings on $\aspec A$ which projective limit over inclusions of open subsets defines a sheaf $\mathcal O_{\aspec A}(U),$ the sheaf associated to the presheaf. The construction is completely natural and gives a fully faithful embedding of integral schemes to associative schemes. We end by proving that an associative scheme is a fine moduli scheme for its simple points.

\end{abstract}

\section{Acknowledgements}
This article could not have been written without the cooperation with Severin Barmeier, who had the idea and insight in noncommutative algebraic geometry necessary for the generalization to noncommutative schemes. Also, I thank Gunnar Fløystad, David Ploog, Eivind Eriksen, Jon Eivind Vatne and Ragni Piene for their comments when lecturing different versions during the winter 2022/2023. As always, my  supervisor O. Arnfinn Laudal has been the greatest inspiration and moderator of this work. Finally, Daniel Larsson corrected me on a couple of issues.

\section{Notation and Preliminaries}\label{SpectralSection}
 
All rings $A$ are assumed to be associative unital rings. For a right $A$-module $M$ we denote the structure morphism by $\eta^A_M:A\rightarrow\enm_{\mathbb Z}(M),$ that is, for $m\in M,a\in A,\ m\cdot a=\eta^A_M(a)(m).$ If we write \emph{module} without prefix, we always mean \emph{right} module.
Recall from \cite{ELS17} that an $A$-module $M$ is \emph{simple} if it has no proper sub-modules. When $k$ is a division ring, $A$ is a $k$-algebra and $M$ is of finite dimension over  $k,$ then $\eta^A_M$ surjective implies $M$ simple, and the converse holds if $\enm_A(M)=k,$ in particular this holds when $k$ is an algebraically closed field.

\section{Deformation Theory}

Let $k$ be a field.

\begin{definition} Let $k$ be in the centre of $R.$ Then the $k^r$-algebra $R$ is called $r$-pointed if it is augmented over $k^r$ such that the diagram $$\xymatrix{k^r\ar[r]^\iota\ar[dr]_{\id}&R\ar[d]^{\pi_R}\\&k^r}$$ commutes. A morphism between $r$-pointed algebras is a $k^r$-algebra homomorphism commuting with the augmentation $\pi_R.$ 
The category  of $r$-pointed $k$-algebras is denoted $\cat{Alg}^r_k$.
\end{definition}

\begin{definition} Two morphisms $\phi,\psi:R\rightarrow S$ in $\cat{Alg}^r_k$ are called smooth equal if the induced linear transformations $\phi_n,\psi_n:R/\ker^n\pi_R\rightarrow S/\ker^n \pi_S$ are equal for all $n\geq 2.$
\end{definition}

Recall from \cite{Lam01} that an associative ring $A$ is called \emph{local} if it has a unique maximal right ideal $\mathfrak m$. Then this ideal is also a left ideal, and this is equivalent to the condition that for every $x\in A$ either $x$ or $1-x$ is a unit.

\begin{lemma}\label{loccompLemma} Let $R$ be a $1$-pointed $k$-algebra, i.e. fitting in the diagram $$\xymatrix{k\ar[dr]_{\id}\ar[r]^\iota&R\ar[d]^\pi\\&k.}$$  If $(\ker\pi)^n=0$ for some $n\geq 1,$ then if $\pi(r)\neq 0,$ then $r$ is a unit in $R.$ In particular, $\mathfrak m=\ker\pi$ is a unique maximal ideal and $R$ is a local ring.
\end{lemma}

\begin{proof} $r\notin\ker\pi\Rightarrow\pi(r)=\alpha=\pi(\iota(\alpha))\Rightarrow r-\iota(\alpha)\in\ker\pi\Rightarrow r=u-x$ where $u=\iota(\alpha)$ is a unit and $x\in\ker\pi.$ Let $t=\sum_{i=1}^n u^{-i}x^{i-1}.$ Then $rt=tr=1.$ Any proper ideal in  $R$ consists of non-units, hence is contained in $\ker\pi.$ Thus this is the unique ideal in $R,$ which is local. 
\end{proof}

\begin{proposition}\label{complProp} Consider a sequence of $1$-pointed $k$-algebras $$\cdots\overset{\pi_n}\rightarrow R_{n}\overset{\pi_{n-1}}\rightarrow R_{n-1}\overset{\pi_{n-2}}\rightarrow\cdots\overset{\pi_{2}}\rightarrow R_2\overset{\pi_{1}}\rightarrow k$$ such that for each $n,$ $(\ker\pi_{n-1})\mathfrak m_n=0,\ \mathfrak m_n=\ker\pi_{n-1}.$ Then $\underset{\underset{n>0}\leftarrow}\lim\  R_n=\hat R$ is a local ring.
\end{proposition}

\begin{proof} By definition, $\hat R$ fits in the diagram $$\xymatrix{k\ar[r]\ar[dr]_\id&\hat R\ar[d]^\pi\\&k.}$$ With the notation from the proof of Lemma \ref{loccompLemma}, it follows that when $r\notin\ker\pi,$ we can put $t=\sum_{n=0}^\infty u^{-(n+1)}x^n\in\hat R$ and then $rt=tr=1.$ When every $r\notin\ker\pi$ is a unit, $\hat R$ is a local ring with (unique) maximal ideal $\ker\pi.$
\end{proof}

\begin{definition} An $r$-pointed ring $\hat{R}$ is called \emph{formal} if it is the projective limit of a sequence of $r$-pointed $k$-algebras $$\cdots\overset{\pi_n}\rightarrow R_{n}\overset{\pi_{n-1}}\rightarrow R_{n-1}\overset{\pi_{n-2}}\rightarrow\cdots\overset{\pi_{2}}\rightarrow R_2\overset{\pi_{1}}\rightarrow k^r$$ such that for each $n,$ $(\ker\pi_{n-1})\mathfrak m_n=0,\ \mathfrak m_n=\ker\pi_{n-1}.$ That is, $\underset{\underset{n>0}\leftarrow}\lim\  R_n=\hat R.$ 
\end{definition}

The following is the most fundamental result in \cite{ELS17} and is proved there.

\begin{theorem}\label{FTD} Let $k$ be a field,  $A$ a $k$-algebra and $\tilde M=\{M_1,\dots,M_r\}$ a set of $r$ right $A$-modules such that $\dim_k(\ext^1_A(M_i,M_j))<\infty$. Then the noncommutative deformation functor has a pro-representing hull $\hat H(\tilde M),$ i.e. there exists a formal $r$-pointed $k$-algebra $\hat H(\tilde M)$ fitting in a diagram \begin{equation}\label{diag1}\xymatrix{A\ar[dr]_{\eta}\ar[r]^-\rho&\hat H(\tilde M)\otimes_{k^r}(\hmm_k(M_i,M_j))=\hat O^A(\tilde M)\ar[d]^\pi\\ &\oplus_{i=1}^r\enm_k(M_i)}\end{equation} such that if $\hat G(\tilde M)$ is another $r$-pointed algebra fitting in the diagram, then there exists a formally unique commuting morphism $\hat H(\tilde M)\rightarrow\hat G(\tilde M).$
\end{theorem}

\begin{remark}\label{remark 1} Let $A$ be a unital associative ring. For $M\in\simp(A)$ we know that $\enm_A(M)$ is a division ring. Assume that the characteristic of $\enm_A(M)$ is equal for all  $M\in\simp(A).$ This says that there exists a number $p$ which is either prime or zero such that $\operatorname{char}(\enm_A(M))=p,\ M\in\simp(A).$
Consider the set of fields $S=\{K|K\hookrightarrow\enm_A(M)\ \forall M\in\simp(A)\}.$ By the condition on $A$ either $\mathbb Z_p$ or $\mathbb Q$ is in $S$ so that the set is non-empty. Ordering $S$ by inclusion of fields, any sequence $\{K_\alpha\}_{\alpha\in I}$ has an upper bound $K=\cup_{\alpha\in I}K_\alpha.$ It follows by Zorn's lemma that there is a choice of at least one maximal such field $K(S).$ Define the algebra $K(S)[A]=\{\alpha a|\alpha\in K(S), a\in A\}$ with multiplication given by  the assumption that $K(S)\subset Z(K(S)[A]).$ For any set $\tilde M$ of simple $A$-modules, Theorem \ref{FTD} applies with the algebra $K(S)[A].$
\end{remark}

Consider the diagram (\ref{diag1}) and let $O_\rho^A(\tilde M)$ be the subalgebra of $\hat O^A(\tilde M)$ generated by $\rho(A)$ together with $\rho(a)^{-1}$ whenever $\rho(a)$ is a unit. It follows from the universal property of $\hat O^A(M)$ that this algebra is independent of the choice of $\rho$ and the following is well defined.

\begin{definition}\label{defTheo} Let $A$ be an associative $k$-algebra and $\tilde M=\{M_1,\dots,M_r\}$ a set of $r$ right $A$-modules. Then $O^A(\tilde M)$ is the subalgebra of $\hat{O}^A(\tilde M)$ generated by $\rho(A)$ together with the set of inverses $\rho(a)^{-1}$ whenever $\rho(a)$ is a unit.
\end{definition}




\begin{example}\label{ComlocEx} Let $\mathfrak p\subset A$ be a prime ideal in a commutative integral domain which is a  $k$-algebra with $k$ algebraically closed. Then we have the composition $A\rightarrow A_{\mathfrak p}\rightarrow\hat A_{\mathfrak p}$ where the first homomorphism is an isomorphism by Krull's Theorem. Thus $O^A(A/\mathfrak p)=A_{\mathfrak p}.$ Then also $O^{A_{\mathfrak p}}(A_{\mathfrak p}/\mathfrak pA_{\mathfrak p})\simeq A_{\mathfrak p}.$
If $\tilde P=\{A/\mathfrak p_i\}_{i=1}^r$ is a set of $r$ prime modules, then $O^A(\tilde P)\simeq\prod_{i=1}^r A_{\mathfrak p_i}.$
\end{example}

\begin{proposition} For a finite family of $A$-modules $\tilde M$, the $k$-algebra $O^A(\tilde M)$ satisfies  $$O^{O^A(\tilde M)}(\tilde M)=O^A(\tilde M).$$ Thus the $O$-construction is a closure operation.
\end{proposition}

\begin{proof} 
This follows by uniqueness of the algebra $\hat{O}^A(\tilde M)$  by the commutativity of the diagram
$$\xymatrix{&O^{O^A(\tilde M)}(\tilde M)\ar[dr]&\\A\ar[ur]\ar[dr]&&\hat{O}^{O^A(\tilde M)}(\tilde M)\\&O^A(\tilde M)\ar[ur]&}$$
\end{proof}

\section{Associative Schemes}

\begin{definition}
Let $f:A\rightarrow B$ be a homomorphism between associative rings and let $M_B$ be a right $B$-module. The $f$-contraction of $M_B$ is the right $A$-module $(M_B)^f_A=f(A)\otimes_A M_B.$
\end{definition}

\begin{lemma}\label{trivialcontractionlemma} Let $\eta^B_{M_B}:B\rightarrow\enm_{\mathbb Z}(M_B)$ be the structure morphism of $M_B.$ Then the structure morphism of $M_A=(M_B)^f_A$ is $\eta^A_{M_A}=\eta^B_{M_B}\circ f.$
\end{lemma}

\begin{example} When $A$ is a commutative, finitely generated $k$-algebra with $k$ algebraically closed, and $M=A/\mathfrak m$ is a simple  $A$-module, we have the diagram $$\xymatrix{A\ar[r]^-f\ar[dr]&O^A(M)=B\ar[d]^{\eta_B}\\&\enm_k(M)}$$ with $B$ local. Thus $B$ has a unique maximal ideal $\mathfrak m\subset B$ such that $M_B\simeq B/\mathfrak m.$ Then we have an isomorphism $A/f^{-1}(\mathfrak m)\overset\sim\rightarrow (M_B)_A^f=f(A)\otimes_A B/\mathfrak m.$
\end{example}

\begin{definition} A right $A$-module $P_A$ is called aprime if there exists a ring $B,$ a simple right $B$-module $M_B$ and a ring homomorphism $f:A\rightarrow B$ such that $P_A=(M_B)^f_A.$
\end{definition}

Notice that a simple module is aprime as a contraction via the identity. When $A$ is a commutative ring, an aprime module $P_A=(B/\mathfrak m)^f_A$ is isomorphic to $A/f^{-1}(\mathfrak m).$  This gives a bijective correspondence between aprime modules and prime ideals.

If $A$ is a commutative ring and $P_A$ is aprime, then an aprime ideal of $P_A$ is a prime ideal $\mathfrak p.$

\begin{lemma} The contraction of an aprime module is aprime.
\end{lemma}

\begin{proof} Let $f:B\rightarrow C$ be a ring homomorphism and $M_C$ simple, such that an aprime module is given as $P_B=(M_C)^f_B=f(B)\otimes_B M_C.$ Let $g:A\rightarrow B$ be a ring homomorphism, then the contraction of $P_B$ is $$(P_B)^g_A=g(A)\otimes_A P_B=g(A)\otimes_A(f(B)\otimes_B M_C)\simeq f(g(A))\otimes_A M_C=(M_C)^{fg}_A.$$ Or we can use Lemma \ref{trivialcontractionlemma} for a more natural proof.
\end{proof}


\begin{definition} We let $\operatorname{aSpec}(A)$ denote the set of right aspectral $A$-modules. For $f\in A$ we  put $D(f)=\{P\in\operatorname{aSpec}(A)|\eta^A_P(f)\text{ is  injective}\}$ and we give $\operatorname{aSpec}(A)$ the topology generated by the sub-basis $\{D(f)\}_{f\in A}.$
\end{definition}

\begin{lemma}\label{continlemma}
Let $\phi:A\rightarrow B$ be a ring homomorphism. Then the induced map $\aspec(\phi):\aspec B\rightarrow\aspec A$ given by $\aspec(\phi)(P)=P^\phi_A$ is continuous.
\end{lemma}

\begin{proof} $\aspec(\phi)^{-1}(D(f))=D(\phi(f)).$
\end{proof}

\begin{definition}\label{topsheafDef} We define a sheaf of rings $\mathcal O_X$ on $X=\operatorname{aSpec}A$ by:

(i) First, define the presheaf $$O_X(U)=\underset{\underset{\tilde M\subseteq U}\leftarrow}\lim O^A(\tilde M)$$ where the limit is taken over all finite subsets $\tilde M$ of simple right $A$-modules. 

(ii) Then, define the sheaf $$\mathcal O_X(U)=\underset{\underset{V \subsetneqq U}\leftarrow}\lim O_X(V).$$
\end{definition}

Let us prove that Definition \ref{topsheafDef} is equivalent to the classical definition of a sheaf (of abelian groups) in algebraic geometry. So let $U$ be an open set covered by the family $\{V_i\}$ and assume that $s\in\mathcal F(U)$ maps to $0=s|_{V_i}\in\mathcal F(V_i)$ for all $i.$ Then the zero homomorphism $\mathbb Z\rightarrow\mathcal F(U)$ for all $V\subsetneq U$ lifts uniquely to $\mathcal F(U)$ so it has to map $1$ to $s=0\in\mathcal F(U).$ Assume that we have elements $s_i\in\mathcal F(V_i)$ for all $i$ such that for all $i,j,$ $s_i|_{V_i\cup V_j}=s_j|_{V_i\cup V_j}\in\mathcal F(U_i\cap U_j).$ The algebra $\mathbb Z[s]$ maps naturally to all $\mathcal F(V_i)$ by sending $s$ to $s_i|V_i,$ and so (by definition of projective limits), there exists a unique homomorphism $\iota:\mathbb Z[s]\rightarrow\mathcal F(U).$ Then $s=\iota(s)\in\mathcal F(U)$ is the unique lifted element.

\begin{definition} For an associative unital ring $A$ we define $\operatorname{aSpec A}$ as the ringed space $(\operatorname{aSpec} A, \mathcal O_{\operatorname{aSpec} A}).$
\end{definition}

\begin{definition} Let $\tilde P=\{P_1,...,P_r\}\subseteq X=\operatorname{aSpec}A$ be a finite set. We define the stalk of any presheaf $F_X$ on $X$ in $\tilde P$ as the inductive limit $$F_{X,\tilde P}=\underset{\underset{\tilde P\subseteq U}\rightarrow}\lim\ F_X(U).$$ 
\end{definition}

\begin{lemma} The stalk in $\tilde P$ of the sheaf $\mathcal O_X$ is isomorphic to the stalk in $\tilde P$ of the presheaf $O_X.$
\end{lemma}

\begin{proof} Inductive limits is a contravariant functor. Thus the morphism $O_X\rightarrow\mathcal O_X$ gives a morphism $\mathcal O_{X,\tilde P}\rightarrow O_{X,\tilde P}.$ Also, for each open $U\supseteq \tilde P$ we have by the same reason a homomorphism $O_{X}(U)\rightarrow \mathcal O_{X,\tilde P}$ giving a homomorphism $O_{X,\tilde P}\rightarrow\mathcal O_{X,\tilde P}.$ By uniqueness, these morphisms are inverses.
\end{proof}

\begin{corollary} If $\tilde P\subseteq X$ is a finite set of aprime right $A$-modules, then $\mathcal O_{X,\tilde P}\simeq O^A({\tilde P}).$
\end{corollary}

\begin{example} It follows from Example \ref{ComlocEx} that for a commutative ring $A$ we have $\operatorname{aSpec}A\simeq\spec A.$
\end{example}

\begin{corollary} Let $A$ be a finite dimensional $\Bbbk$-algebra, $\Bbbk$ algebraically closed. Then the homomorphism  $$\eta:A\rightarrow O^A(\simp A)$$ where $\simp A$ is the set of all simple (right) $A$-modules, is an isomorphism.
\end{corollary}

\begin{proof} When $A$ is finite dimensional $\hat O^A(\tilde M)\simeq O^A(\tilde M)$ for any set of simple modules $\tilde M.$ Also, because $\Bbbk$ is algebraically closed, $\mathcal O^A(\operatorname{aSpec}A)\simeq\mathcal O^A(\simp A).$ 
\end{proof}

\begin{definition}\label{Schemedef} An associative scheme is a topological space with a  sheaf of associative rings $(X,\mathcal O_X)$ which has a cover of open affine subsets $X=\cup_{i\in I}\operatorname{aSpec}(A_i)$ where each $A_i$ is an associative ring. A morphism of associative schemes is  a morphism of ringed spaces. The category of associative schemes  is denoted $\operatorname{aSch}.$
\end{definition}

\begin{definition} An associative variety over an algebraically closed field $\Bbbk$ is an irreducible associative variety over $\Bbbk.$
\end{definition}

\begin{corollary} Let $\phi:A\rightarrow B$ be homomorphism of associative rings. Then there is a natural morphism $\operatorname{aSpec}(\phi):\operatorname{aSpec}B\rightarrow\operatorname{aSpec}A$
\end{corollary}

\begin{proof} The diagram $$\xymatrix{B\ar[drr]_{\eta^B_P}\ar[r]^\phi&A\ar[dr]^{\eta^A_P}\ar[r]^-\rho&A_P\ar[d]^{\eta^{A_P}_P}\\&&\enm_k(P)}$$ gives the map on points sending the aprime $A$-module $P$ to the aprime $B$-module $P.$ The universal property of $A_P$ gives that there is a unique morphism $B_P\rightarrow A_P$ which gives the corresponding morphism of pointed rings, thus a morphism of pre-sheaves, then of sheaves.
\end{proof}

\section{Associative Moduli Schemes}

Consider an associative scheme $X$ as given in Definition \ref{Schemedef}, by which we mean an associative scheme $(X,\mathcal O_X)$ over an algebraically closed field $\Bbbk,$ see Remark \ref{remark 1}. In \cite{ELS17} we prove that the noncommutative deformation functor of sheaves of $\mathcal O_X$-modules has a pro-representing hull. Combining with Definition $\ref{defTheo}$ this gives the following.

\begin{corollary} For a finite set $\mathcal F=\{\mathcal F_1,\dots,\mathcal F_r\}$ of right $\mathcal O_X$-modules, there exists a natural sheaf $\mathcal A_{\mathcal F}$ of $\Bbbk^r$-algebras fitting in the diagram $$\xymatrix{\mathcal O_{X}\ar[dr]_\eta\ar[r]^\rho&\mathcal A_{\mathcal F}\ar[d]^{\oplus\eta_{\mathcal F_i}^{A_{\mathcal F}}}\\&\oplus_{i=1}^r\enm_\Bbbk(\mathcal F_i).}$$ 
\end{corollary}

\begin{proof} Consider the diagram (\ref{diag1}) where we replace $A$ with $\mathcal O_X$ and $\tilde M$ with $\mathcal F.$ Let $O_X(\mathcal F)$ be the sheaf of subalgebras of $\hat{H}(\mathcal F)\otimes_{\Bbbk^r}(\mathcal Hom_k(\mathcal F_i,\mathcal F_j)))=\hat O_X(\mathcal F)$ generated by $\rho(\mathcal O_X)$ together with $\rho(a)^{-1}$ locally whenever $\rho(a)$ is a unit. It follows from the universal property of $\hat O^A(M)$ that $O_X(\mathcal F)$ is independent of the choice of versal morphism $\rho,$ and the natural behaviour follows. Now, let $\mathcal O_X(U)=\underset{\underset{V\subsetneq U}\leftarrow}\lim O_X(V).$
\end{proof}

The associative scheme $X$ is covered by open affine subschemes $U\simeq\aspec A$ so that every point $x\in U\subseteq X$ corresponds to an aprime $\mathcal O_X(U)$-module $M_x.$

\begin{definition} For $x\in X$  the skyscraper sheaf $\mathcal F_x$ in the point $x$ is defined by $$\mathcal F_x=\begin{cases}M_x,&x\in U\\0, &x\notin U\end{cases}$$ where $M_x$ is the $\mathcal O_X(U)$-module corresponding to $x.$
\end{definition}

\begin{definition} $X\in\cat{aSch}/k$ is a fine associative moduli scheme for a functor $F:\cat{aSch}/k\rightarrow\cat{Sets}$ if $\mor_{\cat{aSch}/k}(-,X)\simeq F.$
\end{definition}

A goal of defining associative schemes is to enhance the category of schemes to increase the family of moduli schemes. For an associative scheme $X/\Bbbk$ we have that a simple point, a point corresponding to a simple module, corresponds to a morphism $\aspec \Bbbk\rightarrow X.$ Then the following theorem follows by construction.

\begin{theorem} Any associative scheme $X/\Bbbk,\ \Bbbk$ algebraically closed, is a fine moduli for its $\Bbbk$-rational points.
\end{theorem}

As a stack is not a category, this result solves moduli problems by categorical objects.

\begin{remark} From Remark \ref{remark 1} we see that we can stratify the set of simple modules by their common fields by characteristic.
\end{remark}


\begin{thebibliography}{99.}


\bibitem{ELS17}
E. Eriksen, O. A. Laudal, A. Siqveland,
Noncommutative Deformation Theory. Monographs and Research Notes in Mathematics. CRC Press, Boka Raton, FL, 2017

\bibitem{Lam01}
T. Y. Lam,
A first course in noncommutative rings. Second edition. Graduate Texts in Mathematics, 131. Springer-Verlag, New York, 2001. xx+385 pp. ISBN: 0-387-95183-0 MR1838439

\bibitem{Laudal21} O. A. Laudal, Mathematical models in science, World Scientific Publishing Co. Pte. Ltd., Hackensack, NJ, 2021



\bibitem{S23}
Arvid Siqveland, 
Associative Algebraic Geometry,
World Scientific Publishing Co. Pte. Ltd., Hackensack, NJ, 2023
ISBN: 977-1-80061-354-6






\end{thebibliography}
\end{document}